\documentclass{amsart}
\usepackage[normalem]{ulem}

\usepackage{amssymb,epsfig,mathrsfs,mathpazo}
\usepackage{enumerate}
\usepackage{color}
\usepackage{epsfig,bm}
\usepackage[multiple]{footmisc}
\usepackage{accents} 

\newtheorem{theorem}{Theorem}[section]
\newtheorem{lemma}[theorem]{Lemma}
\newtheorem{algorithm}[theorem]{Algorithm}
\newtheorem{proposition}[theorem]{Proposition}

\theoremstyle{definition}

\theoremstyle{remark}
\newtheorem{remark}[theorem]{Remark}

\numberwithin{equation}{section}

\newcommand{\R}{\mathbb R}

\newcommand{\N}{\mathbb N}

\newcommand{\cL}{\mathcal L}
\newcommand{\Lis}{\cL\mathrm{is}}

\newcommand{\identity}{\mathrm{Id}}

\DeclareMathOperator{\ran}{ran}
\DeclareMathOperator{\supp}{supp}

\DeclareMathOperator*{\argmin}{argmin}   
\DeclareMathOperator{\dom}{dom}

\DeclareMathOperator{\divv}{div}

\newcommand{\be}{\begin{equation}}
\newcommand{\ee}{\end{equation}}

\newcommand{\1}{\mathbb 1}

\newcommand{\tria}{{\mathcal T}}

\newcommand{\TT}{\mathcal T}

\newcommand*\patchAmsMathEnvironmentForLineno[1]{%
  \expandafter\let\csname old#1\expandafter\endcsname\csname #1\endcsname
  \expandafter\let\csname oldend#1\expandafter\endcsname\csname end#1\endcsname
  \renewenvironment{#1}%
     {\linenomath\csname old#1\endcsname}%
     {\csname oldend#1\endcsname\endlinenomath}}%
\newcommand*\patchBothAmsMathEnvironmentsForLineno[1]{%
  \patchAmsMathEnvironmentForLineno{#1}%
  \patchAmsMathEnvironmentForLineno{#1*}}%
\AtBeginDocument{%
\patchBothAmsMathEnvironmentsForLineno{equation}%
\patchBothAmsMathEnvironmentsForLineno{align}%
\patchBothAmsMathEnvironmentsForLineno{flalign}%
\patchBothAmsMathEnvironmentsForLineno{alignat}%
\patchBothAmsMathEnvironmentsForLineno{gather}%
\patchBothAmsMathEnvironmentsForLineno{multline}%
}
\usepackage[mathlines]{lineno}

\title{Further results on a space-time FOSLS formulation of parabolic PDEs}
\date{\today}

\author{Gregor Gantner and Rob Stevenson}

\address{
Korteweg-de Vries (KdV) Institute for Mathematics, University of Amsterdam, P.O. Box 94248, 1090 GE Amsterdam, The Netherlands.
}
\email{g.gantner@uva.nl, r.p.stevenson@uva.nl}
\thanks{The first author has been supported by the Austrian Science Fund (FWF) under grant J4379-N.
The second author has been supported by NSF Grant DMS 172029.}



\begin{document}

\begin{abstract} 
In [2019, Space-time least-squares finite elements for parabolic equations, arXiv:1911.01942] by F\"{u}hrer\& Karkulik, well-posedness of a space-time First-Order System Least-Squares formulation of the heat equation was proven. 
In the present work, this result is generalized to general second order parabolic PDEs with possibly inhomogenoeus boundary conditions, and plain convergence of a standard adaptive finite element method driven by the least-squares estimator is demonstrated. 
The proof of the latter easily extends to a large class of least-squares formulations.

\end{abstract}

\maketitle

\section{Introduction}
Currently, there is a growing interest in simultaneous space-time methods for solving parabolic evolution equations originally introduced in \cite{18.63,18.64}, see e.g.,~\cite{77.5,11,299,249.2,75.27,169.05,247.155,64.577,234.7,243.867,310.6,249.3,75.257}.
Main reasons are that, compared to classical time marching methods, space-time methods 
are much better suited for a massively parallel implementation, are guaranteed to give quasi-optimal approximations from the trial space that is employed, 
have the potential to drive optimally converging simultanously space-time adaptive refinement routines, 
and  they provide enhanced possibilities for reduced order modelling
of parameter-dependent problems.
On the other hand, space-time methods require more storage. This disadvantage however vanishes for problems of optimal control or data assimilation, for which the solution is needed simultaneously over the whole time interval anyway.

The common space-time variational formulation of a parabolic equation results in a bilinear form that is non-coercive.
For the heat equation $\partial_t u -\Delta_{\bf x} u = f$, $u(0,\cdot)=u_0$ on a time-space cylinder $I\times \Omega$, where $I:=(0,T)$ and $\Omega \subset \R^d$,  with homogeneous Dirichlet boundary conditions, the corresponding operator is a boundedly invertible linear mapping between $X$ and $Y' \times L_2(\Omega)$, where $X:=L_2(I;H^1_0(\Omega)) \cap H^1(I;H^{-1}(\Omega))$ and $Y:=L_2(I;H^1_0(\Omega))$.
As a consequence of the non-coercivity, it requires a careful selection of the test space to arrive at a stable Petrov--Galerkin system whose solution is a quasi-best approximation from the trial space.
To relax the conditions on the test space, a minimal residual Petrov--Galerkin discretization was introduced in \cite{11}.
It has an equivalent interpretation as a Galerkin discretization of an extended self-adjoint, indefinite mixed system, with the Riesz lift of the residual of the primal variable from the `trial space' being an additional variable from the `test space'.
In \cite{249.99}, uniform inf-sup stability was demonstrated for both trial and test space being finite element spaces of comparable dimensions, 
w.r.t. general partitions of the space-time cylinder into prismatic elements, which however must be decomposable into `time-slabs'. The latter means that a possibly non-uniform
partition of the time interval must be global in space, which does not align with the aim to permit fully-flexible local refinements in space and time.

In the recent work \cite{75.257} by F\"{u}hrer and Karkulik, for the aforementioned heat equation with forcing term $f \in L_2(I \times \Omega)$ and initial condition $u_0\in L_2(\Omega)$, it was proven that 
with $\widetilde U_{0}:=\{{\bf u}\in X \times L_2(I\times \Omega)^d\colon\divv {\bf u} \in L_2(I\times\Omega)\}$ equipped with the graph norm,
\begin{align*}
\argmin_{{\bf u}=(u_1,{\bf u}_2) \in \widetilde U_{0}} \|{\bf u}_2 +\nabla_{\bf x} u_1\|^2_{L_2(I\times\Omega)^d}+\|\divv {\bf u}_2-f\|_{L_2(I\times\Omega)}^2+\|u(0,\cdot)-u_0\|^2_{L_2(\Omega)}
\end{align*}
is a well-posed First-Order System Least-Squares (FOSLS) formulation for the pair of the solution $u=u_1$ and (minus) its spatial gradient $-\nabla_{\bf x} u={\bf u}_2$.
This formulation can already be found in \cite{23.5} without a proof of its well-posedness though.

The FOSLS formulation from \cite{75.257} has major advantages. The Euler--Lagrange equations resulting from the minimization problem correspond to a symmetric, coercive bilinear form on $\widetilde U_{0} \times \widetilde U_{0}$, so that the Galerkin approximation from \emph{any} conforming trial space is a quasi-best approximation from that space. In other words, there are no issues with stability or restrictions on the partitions of the space-time cylinder underlying the finite element spaces. The minimization is w.r.t.~$L_2$-norms, so that the arising stiffness matrix is computable and sparse and can be easily computed. The least-squares functional provides an a posteriori estimator that is equivalent to the norm on $\widetilde U_{0}$ of the error.
The squared estimator is a sum of squared local error indicators associated to the individual elements, which immediately suggests an adaptive solution method.

Considering general least-squares methods, we mention that  although a least-squares estimator is efficient and reliable, and the resulting adaptive routine is generally observed to converge, even with an optimal rate, 
a proof of ($Q$-linear) convergence of such an adaptive routine has only been given for a FOSLS formulation of Poisson's equation with D\"{o}rfler marking for a bulk parameter that is sufficiently close to $1$, see~\cite{cpb17}.

A disadvantage of the FOSLS method from \cite{75.257} is that the graph norm on $\widetilde U_{0}$ for the error in the pair $(u,-\nabla_{\bf x} u)$ is considerably stronger than the $X$-norm for the error in $u$.
This appears from the low convergence rates reported in \cite{75.257} for the adaptive routine with standard Lagrange finite element spaces applied to non-smooth solutions, e.g.,~as those that result from a discontinuity in the transition of 
 initial and boundary data.  Furthermore, as far as we know, an open problem is the development of optimal preconditioners for the space $\widetilde U_{0}$, which is an important issue in view of the fact that with space-time methods, a PDE posed on a $(d+1)$-dimensional domain has to be solved.
 
 In the current work, we contribute to a further development of the FOSLS method from \cite{75.257}. In particular,
\begin{itemize}
\item we show that  $\widetilde U_{0}$ is isomorphic to $U_0:=\{{\bf u}\in L_2(I;H^1_0(\Omega)) \times L_2(I\times \Omega)^d\colon\divv {\bf u} \in L_2(I\times\Omega)\}$ equipped with the graph norm (Proposition~\ref{equiv}), which circumvents the dual norm incorporated in the definition of $X$.
It is a key ingredient in the derivation of most of the other results from this work;
\item
we show that the FOSLS method applies to general parabolic equations of second order with homogeneous Dirichlet, homogeneous Neumann, or mixed homogeneous Dirichlet and Neumann boundary conditions (Theorem~\ref{thm2} and Proposition~\ref{prop:Bu2least});
\item we extend the FOSLS method to forcing functions $f \not\in L_2(I\times \Omega)$ (Proposition~\ref{prop:Bu2least});
\item by appending an additional term to the least-squares functional measuring the squared error in the boundary data, we extend the FOSLS method to inhomogeneous Dirichlet (Theorem~\ref{thmDir}) or Neumann data (Theorem~\ref{thmNeu}), where, however, the norms in which these errors are measured are not of $L_2$-type;
\item finally, using the framework developed by Siebert (\cite{249.025}), which particularly allows for relatively general marking strategies (Remark~\ref{rem:marking}), we prove plain convergence (Theorem~\ref{thm:error_convergence}) of the adaptive FOSLS method (Algorithm~\ref{alg:adaptive}) for homogeneous Dirichlet boundary conditions driven by the least-squares estimator. This convergence proof generalizes to a large class of least-squares formulations (Remark~\ref{rem:general ls}), including, e.g., the aforementioned FOSLS formulation of the Poisson model problem.
Independently, \cite{fp20} has recently used a similar proof idea to derive convergence of various least-squares formulations, excluding however the considered space-time FOSLS.
\end{itemize}

The remainder of the current section fixes some notation (Subsection~\ref{sec:notation}), recalls abstract parabolic evolution equations (Subsection~\ref{sec:abstract}), and introduces the particular instance of parabolic PDEs of second order (Subsection~\ref{S2ndorder})
 that will be considered throughout the manuscript.

\subsection{Notation}\label{sec:notation}
In this work, by $C \lesssim D$ we will mean that $C$ can be bounded by a multiple of $D$, independently of parameters on which C and D may depend. 
Obviously, $C \gtrsim D$ is defined as $D \lesssim C$, and $C\eqsim D$ as $C\lesssim D$ and $C \gtrsim D$.

For normed linear spaces $E$ and $F$, we will denote by $\cL(E,F)$ the normed linear space of bounded linear mappings $E \rightarrow F$, 
and by $\Lis(E,F)$ its subset of boundedly invertible linear mappings $E \rightarrow F$.
We write $E \hookrightarrow F$ to denote that $E$ is continuously embedded into $F$.
For simplicity only, we exclusively consider linear spaces over the scalar field $\R$.

For a Hilbert space $W$ that is densely and continuously embedded in a space of type $L_2(\Sigma)$, we mostly use the scalar product on $L_2(\Sigma)$ to denote its unique extension to the duality pairing on $W' \times W$.

\subsection{Abstract parabolic evolution equation}
\label{sec:abstract}
Let $V$ and $H$ be separable Hilbert spaces such that $V\hookrightarrow H$ with dense and compact embedding. 
Identifying $H$ with its dual, we obtain the Gelfand triple $V \hookrightarrow H \eqsim H' \hookrightarrow V'$.
For almost all $t \in I:=(0,T)$, let $a(t;\cdot,\cdot)$ be a bilinear form on $V \times V$ such that 
for any $\mu,\lambda\in V$, $t\mapsto a(t;\mu,\lambda)$ is measurable on $I$, and such that 
for some constant $\varrho \geq 0$,
for a.e.~$t \in I$, and all $\mu,\lambda$,
\begin{alignat*}{2} 
|a(t;\mu,\lambda)| 
& \lesssim  \|\mu\|_{V} \|\lambda\|_{V}  \quad &&\text{({\em boundedness})},
\\ \label{2}
 a(t;\mu,\mu)  +\varrho \|\mu\|^2 &\gtrsim \|\mu\|^2_{V}  \quad &&\text{({\em G{\aa}rding inequality})}.
\end{alignat*}
With $(A(t)\cdot)(\cdot):=a(t;\cdot,\cdot)$, we consider the \emph{parabolic initial value problem} of finding $u\colon I \rightarrow V$ such that
\begin{align*}
\left\{
\begin{array}{rl} 
\frac{d u}{d t}(t) +A(t) u(t)&\!\!\!= g(t) \quad\text{for a.e. } t \in I,\\
u(0) &\!\!\!= u_0.
\end{array}
\right.
\end{align*}

A proof of the following result is found in \cite{247.15}, see also \cite[Chapter~IV, \S26]{314.9} and \cite[Chapter~XVIII, \S3]{63}.

\begin{theorem}\label{thm1}
With $X:=L_2(I;{V}) \cap H^1(I;V')$, $Y:=L_2(I;{V})$,
\begin{align*}
(Bu)(v):= \int_I \Big\{ 
(\partial_t u(t,\cdot))(v(t,\cdot)) 
+
a(t;u(t),v(t)) \Big\}dt,
\end{align*}
and $\gamma_0:= u\mapsto u|_{t=0}$, it holds that
\begin{align*}
\left[\begin{array}{@{}c@{}} B \\ \gamma_0\end{array} \right]\in \Lis\big(X,(Y \times H)'\big),
\end{align*}
with upper bounds for the norm of the operator and that of its inverse 
only dependent on upper bounds for the boundedness constant, the reciprocal of the constant in the G{\aa}rding inequality, and $\varrho$.
\end{theorem}

So for $(g,u_0) \in Y' \times H$, a well-posed variational formulation of the parabolic problem reads as finding $u \in X$ such that 
$(Bu,\gamma_0 u)=(g,u_0)$.

\subsection{Parabolic equations of second order} \label{S2ndorder}
For a bounded Lipschitz domain $\Omega \subset \R^d$ with outer normal ${\bf n}_{\bf x}\in\R^d$,
relatively open subsets $\Gamma_D$ and $\Gamma_N$ of $\partial\Omega$ with $\Gamma_D \cap \Gamma_N=\emptyset$ and $\overline{\Gamma_D \cup \Gamma_N}=\partial\Omega$, ${\bf b} \in L_\infty(I \times \Omega)^d$, $c \in L_\infty(I \times \Omega)$, and ${\bf A}={\bf A}^\top \in L_\infty(I \times \Omega)^{d\times d}$ uniformly positive definite, we consider the problem of finding $u\colon I \times \Omega \rightarrow \R$ that for given data $f$, $\phi$, $u_D$, and $u_0$ satisfies 
\be \label{2nd}
\left\{
\begin{array}{rcll}
\partial_t u - \divv_{\bf x}
{\bf A} \nabla_{\bf x} u+{\bf b}\cdot \nabla_{\bf x} u + c u & = & f & \text{ on } I \times \Omega,\\
({\bf A} \nabla_{\bf x} u)\cdot {\bf n}_{\bf x} & = & \phi & \text{ on }I \times \Gamma_N,\\
u & = & u_D & \text{ on }I \times \Gamma_D,\\
u(0,\cdot) & = & u_0 & \text{ on }\Omega.
\end{array}
\right.
\ee
Taking until Section~\ref{sec:inhom} a homogeneous Dirichlet datum  $u_D=0$, a variational formulation of \eqref{2nd} leads to a problem as in Theorem~\ref{thm1}, where
$V:=H^1_{D}(\Omega)=\{u \in H^1(\Omega)\colon u|_{\Gamma_D}=0 \}$ and $H:=L_2(\Omega)$, so that
\begin{align*}
X=L_2(I;H_D^1(\Omega)) \cap H^{1}(I;H_D^1(\Omega)'), \qquad Y=L_2(I;H_D^1(\Omega)),
\end{align*}
the bilinear form reads as
\begin{align*}
a(t;\mu,\lambda):=\int_\Omega {\bf A}(t,{\bf x})\nabla \mu({\bf x}) \cdot \nabla \lambda({\bf x})+({\bf b}(t,{\bf x}) \cdot \nabla \mu({\bf x}) + c(t,{\bf x}) \mu({\bf x})) \lambda({\bf x}) \,d{\bf x},
\end{align*}
and the forcing term reads as
\begin{align}\label{eq:RHS g}
g(v):=\int_{I \times \Omega} f v \,d{\bf x}\,dt+\int_{I \times \Gamma_N} \phi v \,d {\bf s}.
\end{align}
As follows from Theorem~\ref{thm1}, this variational problem is actually well-posed for \emph{any} $g \in Y'$. 
For a discussion in which sense the solution of the variational problem can be interpreted as a solution of \eqref{2nd}, we refer to \cite[pages 524--528]{63}.


Concerning the bilinear form $a$, both its boundedness constant, the reciprocal of the constant in the G{\aa}rding inequality, and $\varrho$ can be bounded in terms of upper bounds for $\|{\bf b}\|_{L_\infty(I \times \Omega)^d}$, $\|c\|_{L_\infty(I \times \Omega)}$, $\|{\bf A}\|_{L_\infty(I \times \Omega)^{d\times d}}$, and $\|{\bf A}^{-1}\|_{L_\infty(I \times \Omega)^{d\times d}}$.

\section{Formulation as a first-order system}

\subsection{Homogeneous boundary conditions}

For the case that $g \in L_2(I\times \Omega)$, we will derive a system for ${\bf u}=(u_1,{\bf u}_2)$ $=$ $(u,-{\bf A}\nabla_{\bf x} u)$ with $u$ being the solution of the variational problem $(Bu,\gamma_0 u)=(g,u_0)$ from Section~\ref{S2ndorder}. Recall that such a problem arises from \eqref{2nd} when 
besides $u_D=0$, it holds that $f \in L_2(I \times \Omega)$ and $\phi=0$.
Generally at the expense of having to solve an additional (elliptic) PDE, general $g \in Y'$ (i.e. $f \not\in L_2(I \times\Omega)$ and/or Neumann datum $\phi \neq 0$) will be handled as well.

Let
\begin{align*}
U:=\{{\bf u}=(u_1,{\bf u}_2)\in L_2(I;H^1(\Omega)) \times L_2(I\times \Omega)^d\colon\divv {\bf u} \in L_2(I\times\Omega)\}
\end{align*}
equipped with graph norm
\begin{align}\label{eq:U norm}
\|{\bf u}\|_U^2 := \|u_1\|_{L_2(I;H^1(\Omega))}^2+\|{\bf u}_2\|_{L_2(I;L_2(\Omega)^d)}^2+\|\divv {\bf u}\|_{L_2(I\times\Omega)}^2.
\end{align}
Knowing that $\divv\colon L_2(I\times \Omega)^{d+1}\supset \dom(\divv) \rightarrow L_2(I\times\Omega)$ is a closed linear operator (a necessary condition for $H(\divv;I\times \Omega)$ being a Hilbert space), from $L_2(I;H^1(\Omega)) \times L_2(I\times \Omega)^d \hookrightarrow L_2(I\times \Omega)^{d+1}$, it  follows that
$\divv\colon L_2(I;H^1(\Omega)) \times L_2(I\times \Omega)^d\supset \dom(\divv) \rightarrow L_2(I\times\Omega)$ is a closed linear operator. Together with the facts that $L_2(I;H^1(\Omega)) \times L_2(I\times \Omega)^d$ and $L_2(I\times\Omega)$ are Hilbert spaces, this shows that
$U$ is a Hilbert space.

With ${\bf n}=(n_t,{\bf n}_{\bf x})$ denoting the outer normal vector on the boundary of $I \times \Omega$,
using that ${\bf u} \mapsto {\bf u}|_{I \times \Gamma_N}\cdot{\bf n}
\in \cL\big(H(\divv;I \times \Omega), H_{00}^{\frac12}(I \times \Gamma_N)'\big)$ 
we define the closed subspace $U_0$ of $U$ by
\begin{align*}
U_0:=\{{\bf u}\in L_2(I;H_D^1(\Omega)) \times L_2(I\times \Omega)^d\colon\divv {\bf u} \in L_2(I\times\Omega),\,
{\bf u}|_{I \times \Gamma_N}\cdot{\bf n}=0\}.
\end{align*}

We start with showing that $U_0$ is isomorphic to a seemingly smaller space that was employed in \cite{75.257}.
\begin{proposition} \label{equiv} It holds that
\begin{align*}
U_{0} \eqsim \widetilde U_{0}:=\{{\bf u}\in X \times L_2(I\times \Omega)^d\colon\divv {\bf u} \in L_2(I\times\Omega),\,
{\bf u}|_{I \times\Gamma_N}\cdot{\bf n}=0\},
\end{align*}
equipped with the graph norm
\begin{align*}
\|{\bf u}\|_{{\widetilde U}_0}^2:=\|u_1\|_{L_2(I;H^1(\Omega))}^2+\|\partial_t u_1\|_{L_2(I;H^1_D(\Omega)')}^2+\|{\bf u}_2\|_{L_2(I;L_2(\Omega)^d)}^2+\|\divv {\bf u}\|_{L_2(I\times\Omega)}^2.
\end{align*}
\end{proposition}%

\noindent This proposition is a direct consequence of the following lemma.

\begin{lemma} \label{pre-equiv}
For ${\bf u} \in H_{0,I \times \Gamma_N}(\divv;I \times \Omega):=\{{\bf u} \in H(\divv;I \times \Omega)\colon {\bf u}|_{I \times \Gamma_N}\cdot{\bf n} =0\}$,  it holds that $\partial_t u_1 \in L_2(I;H^1_D(\Omega)')$ with
\begin{align*}
\|\partial_t u_1\|_{L_2(I;H^1_D(\Omega)')} \leq \|{\bf u}\|_{H(\divv;I \times \Omega)}.\end{align*}
\end{lemma}

\begin{proof}
For smooth ${\bf u} \in H_{0,I \times \Gamma_N}(\divv;I \times \Omega)$ (for which ${\bf u}\cdot {\bf n}$ is defined in the classical pointwise sense), we have  $\divv {\bf u}=\partial_t u_1+\divv_{\bf x} {\bf u}_2$. For smooth $v \in L_2(I;H^1_D(\Omega))$
we have 
\begin{align*}
\int_{I \times \Omega} {\bf u}_2 \cdot \nabla_{\bf x} v \,d{\bf x}\,dt&=
-\int_{I \times \Omega} v  \divv_{\bf x}{\bf u}_2 \,d{\bf x}\,dt+\int_{I \times \Gamma_N}   {\bf u}_2\cdot {\bf n}_{\bf x} \,v\,ds
\\&=-\int_{I \times \Omega} v  \divv_{\bf x}{\bf u}_2 \,d{\bf x}\,dt +\int_{I \times \Gamma_N} {\bf u}\cdot {\bf n}\,v \,ds\\
&=-\int_{I \times \Omega} v  \divv_{\bf x}{\bf u}_2 \,d{\bf x}\,dt.
\end{align*}
Since the set of such $v$ is dense in $L_2(I;H^1_D(\Omega))$, we conclude
\begin{align*}
\|\partial_t u_1\|_{L_2(I;H^1_D(\Omega)')}
&\le \|\divv {\bf u}\|_{L_2(I;H^1_D(\Omega)')} + \|\divv_{\bf x} {\bf u_2} \|_{L_2(I;H^1_D(\Omega)')} 
\\
&\leq \|\divv {\bf u}\|_{L_2(I\times \Omega)}+\|{\bf u}_2\|_{L_2(I\times\Omega)^d} \leq \|{\bf u}\|_{H(\divv;I \times \Omega)}.
\end{align*}
Since the set of such  ${\bf u}$ is dense in $H_{0,I \times \Gamma_N}(\divv;I \times \Omega)$, the proof is completed.
For $\partial (I\times\Omega)$ instead of $I\times\Gamma_N$, the corresponding density result is well-known. 
The proof~\cite[Theorem~2.6]{75.5} easily generalizes to $I\times\Gamma_N$ using that the term $l_{d+2}\in H^1(I\times\Omega)$ from there additionally satisfies that $l_{d+2}|_{\partial(I\times\Omega)\setminus\overline{I\times\Gamma_N}}=0$ as $l_{d+2}|_{\partial(I\times\Omega)}$ is orthogonal to ${\bf u}\cdot {\bf n}$ for all smooth ${\bf u} \in H_{0,I \times \Gamma_N}(\divv;I \times \Omega)$.
\end{proof}

The following theorem generalizes~\cite{75.257}, see Remark~\ref{litt} for a discussion.

\begin{theorem}[homogeneous Dirichlet] \label{thm2}
It holds that
\begin{align*}
G\colon (u_1,{\bf u}_2) \mapsto ({\bf u}_2+{\bf A} \nabla_{\bf x} u_1,\divv {\bf u}-{\bf b} \cdot {\bf A}^{-1}  {\bf u}_2 + c u_1 ,u_1(0,\cdot))\\
 \in \Lis(U_{0},L_2(I\times\Omega)^d \times L_2(I\times\Omega) \times L_2(\Omega)).
\end{align*}
\end{theorem}

\begin{remark}\label{rem:grad term}
Analogously, one can prove the same result for
$(u_1,{\bf u}_2) \mapsto ({\bf u}_2+{\bf A} \nabla_{\bf x} u_1,\divv {\bf u}+{\bf b} \cdot \nabla_{\bf x}  u_1 + c u_1 ,u_1(0,\cdot))$.
\end{remark}

\begin{proof} Boundedness of $G$ follows from the definition of $U_0$, and the fact that $X \hookrightarrow C(\bar{I};L_2(\Omega))$ (\cite[Chapter~1, Theorem~3.1]{185}) in combination with Proposition~\ref{equiv}.

As we have seen in the proof of Lemma~\ref{pre-equiv},
 for ${\bf u} \in U_0$ and $v \in L_2(I;H^1_D(\Omega))$, it holds that  $(-\nabla_{\bf x}'{\bf u}_2)(v)=\int_{I \times \Omega} v \divv_{\bf x} {\bf u}_2 \,d{\bf x}\,dt$.
From Theorem~\ref{thm1} we infer that
\begin{align*}
\|u_1\|_{L_2(I;H^1(\Omega))} \leq \|u_1\|_{X} 
\lesssim 
\|B u_1\|_{L_2(I;H^1_D(\Omega)')}+\|u_1(0,\cdot)\|_{L_2(\Omega)},
\end{align*}
where
\begin{align*}
&\|B u_1\|_{L_2(I;H^1_D(\Omega)')} = \|\partial_t u_1+\nabla_{\bf x}' {\bf A}\nabla_{\bf x} u_1 +{\bf b} \cdot \nabla_{\bf x} u_1 + c u_1\|_{L_2(I;H^1_D(\Omega)')}\\
& \le \|\partial_t u_1-\nabla_{\bf x}' {\bf u}_2  +{\bf b} \cdot \nabla_{\bf x} u_1 + c u_1\|_{L_2(I;H^1_D(\Omega)')}+\|\nabla_{\bf x}'({\bf u}_2+{\bf A}\nabla_{\bf x} u_1)\|_{L_2(I;H^1_D(\Omega)')}\\
& \lesssim \|\divv {\bf u} +{\bf b} \cdot \nabla_{\bf x} u_1 + c u_1\|_{L_2(I \times \Omega)}+\|{\bf u}_2+{\bf A}\nabla_{\bf x} u_1\|_{L_2(I\times \Omega)^d}\\
& \lesssim \|\divv {\bf u} -{\bf b} \cdot {\bf A}^{-1}{\bf u}_2 + c u_1\|_{L_2(I\times\Omega)}+\|{\bf u}_2+{\bf A}\nabla_{\bf x} u_1\|_{L_2(I\times \Omega)^d}.
\end{align*} 
From
\begin{align*}
\|{\bf u}_2\|_{L_2(I\times\Omega)^d} &\leq \|{\bf u}_2 +{\bf A} \nabla_{\bf x}u_1\|_{L_2(I\times\Omega)^d}+\|{\bf A}\nabla_{\bf x}u_1\|_{L_2(I\times\Omega)^d}\\
& \lesssim \|{\bf u}_2+{\bf A} \nabla_{\bf x}u_1\|_{L_2(I\times\Omega)^d}+\|u_1\|_{L_2(I;H^1(\Omega))},
\intertext{and}
\| \divv {\bf u} \|_{L_2(I\times\Omega)} &\!\lesssim \| \divv {\bf u}\!-\!{\bf b} \cdot {\bf A}^{-1}{\bf u}_2 \!+\! c u_1\|_{L_2(I\times\Omega)} \!+\!\|{\bf u}\|_{L_2(I \times \Omega)^{d+1}},
\end{align*}
we conclude that $\|{\bf u}\|_U \lesssim \|G{\bf u}\|_{L_2(I\times \Omega)^d \times L_2(I\times\Omega) \times L_2(\Omega)}$, and thus in particular that $G$ is injective. 

Given $({\bf q},h,u_0) \in L_2(I\times\Omega)^d \times L_2(I\times \Omega) \times L_2(\Omega)$, let $u_1 \in X$ be the solution of 
\begin{align*}
\left[\begin{array}{@{}c@{}} B \\ \gamma_0\end{array} \right]u_1=\left[\begin{array}{@{}c@{}} v \mapsto  \int_{I \times \Omega} (h+{\bf b}\cdot {\bf A}^{-1} {\bf q}) v +{\bf q} \cdot \nabla_{\bf x} v  \,d{\bf x}\,dt\\ u_0\end{array} \right] \in L_2(I;H^1_D(\Omega)') \times L_2(\Omega),
\end{align*}
so that for $v \in L_2(I;H^1_D(\Omega))$
\begin{align*}
\int_{I \times \Omega} \partial_t u_1\,v+{\bf A} \nabla_{\bf x} u_1\cdot \nabla_{\bf x} v+{\bf b}\cdot\nabla_{\bf x} u_1\, v + c\,u_1\,v \,d{\bf x}\,dt
\\
\quad=
\int_{I \times \Omega} (h+{\bf b}\cdot {\bf A}^{-1} {\bf q}) v +{\bf q} \cdot \nabla_{\bf x} v  \,d{\bf x}\,dt,
\end{align*}
and thus for ${\bf u}_2:={\bf q}-{\bf A} \nabla_{\bf x}u_1 \in L_2(I\times \Omega)^d$
\begin{align*}
\int_{I \times \Omega} \partial_t u_1\,v-{\bf u}_2 \cdot\nabla_{\bf x} v \,d{\bf x}\,dt=
\int_{I \times \Omega} \underbrace{(h+{\bf b}\cdot {\bf A}^{-1} {\bf u}_2- c u_1)}_{=:\tilde h \in L_2(I \times \Omega)}v \,d{\bf x}\,dt .
\end{align*}
For $v \in H^1(I\times \Omega)$ that vanish at $\partial (I \times \Omega) \setminus \overline{I \times \Gamma_N}$, one has $\int_{I \times \Omega} \partial_t u_1 \,v\,d{\bf x}\,dt$ $=$ $-\int_{I \times \Omega}u_1  \partial_t  v\,d{\bf x}\,dt$, and therefore $\int_{I \times \Omega} \partial_t u_1\,v-{\bf u}_2 \cdot\nabla_{\bf x} v \,d{\bf x}\,dt=-\int_{I \times \Omega} {\bf u}\cdot \nabla v\,d{\bf x}\,dt$, which shows $\divv {\bf u}=\tilde h$.
Moreover, for such $v$, it holds that
\begin{align*}
\int_{I \times \Gamma_N}  {\bf u}\cdot{\bf n}\,v\,ds&=
\int_{I \times \Omega} v \divv {\bf u}+{\bf u}\cdot \nabla v \,d{\bf x}\,dt=
\int_{I \times \Omega} v \tilde h+u_1 \partial_t v +{\bf u}_2\cdot \nabla_{\bf x} v \,d{\bf x}\,dt\\
&=
\int_{I \times \Omega} v \tilde h-(\partial_t u_1 \, v -{\bf u}_2\cdot \nabla_{\bf x} v) \,d{\bf x}\,dt=0,
\end{align*}
which proves that ${\bf u}|_{I \times \Gamma_N}\cdot {\bf n}=0$, and so ${\bf u} \in U_0$. 
We conclude that $G{\bf u}=({\bf q},h,u_0)$, i.e., $G$ is surjective, which completes the proof.
\end{proof}

Next, using Theorem~\ref{thm2}, we show that the well-posed standard variational formulation of the parabolic problem discussed in Subsections~\ref{sec:abstract}--\ref{S2ndorder}, thus with homogeneous Dirichlet datum $u_D=0$, has an equivalent formulation as a well-posed first-order system.
As a preparation, we note that any forcing term $g \in L_2(I;H_D^1(\Omega)')$ can (non-uniquely) be written in the form
\be \label{splitting}
g(v)=\int_{I \times \Omega} g_1 v +{\bf g}_2 \cdot \nabla_{\bf x} v \,d{\bf x}\,dt \quad \text{for all } v \in L_2(I;H^1_D(\Omega)),
\ee
for some $g_1 \in L_2(I;L_2(\Omega))$ and ${\bf g}_2 \in L_2(I;L_2(\Omega)^d)$.
Take, e.g., $g_1=w$ and ${\bf g}_2=\nabla_{\bf x} w$ with $w \in L_2(I;H^1_D(\Omega))$ being the Riesz lift of $g$ defined by
\be \label{lift}
\int_{I \times \Omega} w v+\nabla_{\bf x} w \cdot \nabla_{\bf x} v\,d{\bf x}\,dt =g(v) \quad \text{for all } v \in L_2(I;H^1_D(\Omega)).
\ee

\begin{proposition} \label{prop:Bu2least}
With a splitting of $g \in\ L_2(I;H^1_D(\Omega)')$ as in \eqref{splitting}, where $(g_1,{\bf g}_2) \in L_2(I;L_2(\Omega))\times L_2(I;L_2(\Omega)^d)$, and $u_0\in L_2(\Omega)$, it holds that $u_1\in X=L_2(I;H_D^1(\Omega)) \cap H^{1}(I;H_D^1(\Omega)')$ solves 
$(B u_1, \gamma_0 u_1)=(g,u_0)$ and  ${\bf u}_2=-{\bf A} \nabla_{\bf x} u_1+{\bf g}_2$
if and only if ${\bf u}=(u_1,{\bf u}_2) \in U_0$ solves
\begin{align*}
G {\bf u}=({\bf g}_2,  g_1 -{\bf b}\cdot {\bf A}^{-1} {\bf g}_2, u_0).
\end{align*}
\end{proposition}

\begin{proof}
With ${\bf u}_2=-{\bf A} \nabla_{\bf x} u_1+{\bf g}_2$, i.e., $(G {\bf u})_1={\bf g}_2$, the equation $B u_1=g$, i.e.,
\begin{align*}
\int_{I \times \Omega} (\partial_t u_1 +{\bf b}\cdot \nabla_{\bf x} u_1 +c u_1) v +{\bf A} \nabla_{\bf x} u_1\cdot \nabla_{\bf x}v \,d{\bf x}\,dt=g(v) \quad (v \in L_2(I;H^1_D(\Omega))),
\end{align*}
is equivalent to
\begin{align} \label{eq:g_tilde}
\int_{I \times \Omega}  \partial_t u_1 v-{\bf u}_2 \cdot \nabla_{\bf x} v \,d{\bf x}\,dt=
 \int_{I \times \Omega} \underbrace{({\bf b}\cdot {\bf A}^{-1}({\bf u}_2-{\bf g}_2)-c u_1+g_1)}_{=:\tilde g \in L_2(I \times \Omega)} v\,d{\bf x}\,dt 
 \\ (v \in L_2(I;H^1_D(\Omega))).\notag
\end{align}
As we have seen in the last paragraph of the proof of Theorem~\ref{thm2}, \eqref{eq:g_tilde} implies $\divv {\bf u}=\tilde g$, i.e., $(G{\bf u})_2= g_1 -{\bf b}\cdot {\bf A}^{-1} {\bf g}_2$, and ${\bf u}|_{I \times \Gamma_N}=0$.

Conversely, let ${\bf u} \in U_0$ satisfy $G {\bf u}=({\bf g}_2,  g_1 -{\bf b}\cdot {\bf A}^{-1} {\bf g}_2, u_0)$. Then, Proposition~\ref{equiv} shows that $u_1 \in X$. Since $\divv {\bf u}=\tilde g$, it remains to show that 
\begin{align*}
 \int_{I \times \Omega}  \partial_t u_1 v-{\bf u}_2 \cdot \nabla_{\bf x} v \,d{\bf x}\,dt=\int_{I \times \Omega}v \divv {\bf u} \,d{\bf x}\,dt \quad(v \in L_2(I;H^1_D(\Omega))).
\end{align*}
The latter relation is already valid for arbitrary ${\bf u} \in H_{0,I\times \Gamma_N}(\divv;I\times \Omega)$ and $v \in L_2(I;H^1_D(\Omega))$. Indeed, for smooth ${\bf u}$ and $v$ in these spaces, it follows by integration by parts, and so by using Lemma~\ref{pre-equiv}, it follows by the density of the sets of those functions in these spaces.
\end{proof}

When $f \in L_2(I \times \Omega)$ and $\phi=0$ in \eqref{2nd}, one has $g=f \in  L_2(I \times \Omega)$ and one obviously  takes $(g_1,{\bf g}_2)=(g,0)$ in the previous proposition.
For $g \in L_2(I;H^1_D(\Omega)')\setminus L_2(I \times \Omega)$ (i.e., $f \not\in L_2(I\times\Omega)$ and/or $\phi \neq 0$)
 generally the splitting of $g$ requires solving~\eqref{lift}.
For the case that $\Gamma_N=\partial\Omega$, an alternative approach for inhomogeneous Neumann datum $\phi \neq 0$ will be presented in Theorem~\ref{thmNeu}.

\begin{remark} \label{litt} Theorem~\ref{thm2} extends the crucial result from \cite{75.257}. 
For the case that ${\bf A}=\identity$, ${\bf b}=0=c$, and $\Gamma_D=\partial\Omega$, there it was shown that the harmlessly different operator $\widetilde{G}\colon {\bf u} \mapsto G(u,-{\bf u}_2)\colon \widetilde{U}_0 \mapsto L_2(I;L_2(\Omega)^d) \times L_2(I;L_2(\Omega)) \times L_2(\Omega)$  is 
in $\Lis(\widetilde{U}_0,\ran \widetilde{G})$, and that $\ran \widetilde{G} \supseteq \{{\bf 0}\} \times L_2(I;L_2(\Omega)) \times L_2(\Omega)$. We showed that $G$, and thus $\widetilde{G}$, is also surjective.
Notice that for well-posedness of a least-squares formulation, this surjectivity is not required. Indeed, bounded invertibility of the operator between its domain and its range is equivalent to boundedness and coercivity of the bilinear form corresponding to the Euler--Lagrange equations resulting from the least-squares functional.

Our motivation to replace ${\bf u}_2$ by $-{\bf u}_2$ is that $\partial_t u_1+\divv_{\bf x} {\bf u}_2$ is the divergence of the vector field ${\bf u}\colon I \times \Omega \rightarrow \R^{d+1}$. 
When imposing, as we do, that the latter divergence is in $L_2(I \times \Omega)$, we know that ${\bf u}$ has a normal trace at $\partial (I \times \Omega)$, which allowed an easy extension to homogeneous Neumann boundary conditions.
Furthermore, in Proposition~\ref{equiv}, we made the observation that $\widetilde{U}_0 \eqsim U_0$, which freed ourselves from the dual norm which is part of the definition of $\widetilde U_0$.
This will also play an essential role in the proofs of Theorem~\ref{thmDir} and \ref{thmNeu} dealing with inhomogeneous boundary conditions, and that of Theorem~\ref{thm:error_convergence} concerning plain convergence of a standard adaptive algorithm.
\end{remark}

\subsection{Inhomogeneous boundary conditions} \label{sec:inhom}

We extend the first-order formulation to cover both inhomogeneous (pure) Dirichlet boundary conditions and inhomogeneous (pure) Neumann boundary conditions, the latter now without the need to compute a Riesz lift of the boundary datum.

The following lemma is essentially a slight generalization of \cite[Theorem~2.1]{249.96}. 
Thinking of $S$ as being a trace operator, it shows how to append (essential) inhomogeneous boundary conditions to an equation that is well-posed for the corresponding homogeneous boundary conditions.

\begin{lemma} \label{lem:bi}
Let $\mathcal{X}$ and $\mathcal{Y}_2$ be Banach spaces, and $\mathcal{Y}_1$ be a normed linear space.
Let $S \in \cL(\mathcal{X},\mathcal{Y}_2)$ be surjective, let $F \in \cL(\mathcal{X},\mathcal{Y}_1)$ be such that 
with $\mathcal{X}_0:=\{x \in \mathcal{X}\colon S x=0\}$,  $F|_{\mathcal{X}_0} \in \Lis(\mathcal{X}_0,\mathcal{Y}_1)$.
Then, $\left[\begin{array}{@{}c@{}} F \\ S\end{array} \right]\in \Lis\big(\mathcal{X},\mathcal{Y}_1\times \mathcal{Y}_2\big)$.
\end{lemma}

\begin{proof} Knowing that $S$ maps the open unit ball of $\mathcal{X}$ onto an open neighborhood of $0 \in \mathcal{Y}_2$ (according to the open mapping theorem),
there exists a constant $r>0$ such that for any $y \in \mathcal{Y}_2$ there exists an $x \in \mathcal{X}$ with $Sx=y$ and $\|x\|_{\mathcal{X}} \leq r \|y\|_{\mathcal{Y}_2}$.
Denoting this mapping $y\mapsto x$ by $E$, from $\ran(\identity - ES) \subseteq\mathcal{X}_0$ we have for $x \in \mathcal{X}$
\begin{align*}
\|x\|_{\mathcal{X}} &\leq
\|E S x\|_{\mathcal{X}}+\|(\identity-E S) x\|_{\mathcal{X}}
\lesssim \|S x \|_{\mathcal{Y}_2}+\|F(\identity-E S) x\|_{\mathcal{Y}_1}\\
&\leq \|S x \|_{\mathcal{Y}_2}+\|F x\|_{\mathcal{Y}_1}+\|FE S x\|_{\mathcal{Y}_1}
\lesssim \|S x \|_{\mathcal{Y}_2}+\|F x\|_{\mathcal{Y}_1} \lesssim \|x\|_{\mathcal{X}}.
\end{align*}
Given $(y_1,y_2) \in \mathcal{Y}_1\times \mathcal{Y}_2$, let $x_2 \in \mathcal{X}$ be such that $S x_2 =y_2$, and $x_0 \in \mathcal{X}_0$ be such that $F x_0=y_1-F x_2$.
Then, $\left[\begin{array}{@{}c@{}} F \\ S\end{array} \right] (x_0+x_2)=\left[\begin{array}{@{}c@{}} y_1 \\ y_2\end{array} \right]$ showing that $\left[\begin{array}{@{}c@{}} F \\ S\end{array} \right]$ is surjective, which completes the proof.
\end{proof}%

In combination with Theorem~\ref{thm2}, Lemma~\ref{lem:bi} allows to prove the following theorem for  inhomogeneous pure Dirichlet boundary conditions.
\begin{theorem}[Inhomogeneous (pure) Dirichlet] \label{thmDir} 
It holds that
\begin{align*}
&G_D\colon {\bf u}=(u_1,{\bf u}_2) \mapsto ({\bf u}_2+{\bf A} \nabla_{\bf x} u_1,\divv {\bf u}-{\bf b} \cdot {\bf A}^{-1}  {\bf u}_2 + c u_1 ,u_1(0,\cdot), u_1|_{I \times \partial\Omega})\\
 &\in \Lis\Big(U,L_2(I\times\Omega)^d \times L_2(I\times\Omega) \times L_2(\Omega) \times \big(L_2(I;H^{\frac12}(\partial\Omega))\cap H^{\frac14}(I;L_2(\partial\Omega))\big)\Big).
\end{align*}
\end{theorem}

\begin{proof} 
An application of Lemma~\ref{pre-equiv} for $\Gamma_N=\emptyset$ shows that for ${\bf u}=(u_1,{\bf u}_2) \in U$, 
\be \label{eq:boundu1}
\|u_1\|_{L_2(I;H^1(\Omega)) \cap H^1(I;H^{-1}(\Omega))} \lesssim \|{\bf u}\|_{U}.
\ee
We will combine this observation with the fact that
\begin{align}\label{eq:Hhalf}
L_2(I;H^1(\Omega)) \cap H^1(I;H^{-1}(\Omega)) \hookrightarrow C(\bar{I};L_2(\Omega)) \cap
H^{\frac12}(I;L_2(\Omega)),
\end{align}
which follows from $[H^{-1}(\Omega),H^1(\Omega)]_{\frac12}=L_2(\Omega)$, see, e.g.~\cite[pages~480 \& 494]{63}.
As shown in \cite[Chapter~4, Theorem~2.1]{185.1},
\begin{align*}
u_1 \mapsto u_1|_{I \times \partial\Omega} \in \cL\Big(L_2(I;H^1(\Omega)) \cap H^{\frac12}(I;L_2(\Omega)),L_2(I;H^{\frac12}(\partial\Omega))\cap H^{\frac14}(I;L_2(\partial\Omega))\Big).
\end{align*}
Together with \eqref{eq:boundu1}--\eqref{eq:Hhalf}, it shows that $G_D$ is bounded.

Since in the current case of $\Gamma_D=\partial\Omega$, we have $\{{\bf u} \in U \colon u_1|_{I \times \partial\Omega}=0\}=U_0$, knowing the result of Theorem~\ref{thm2}, Lemma~\ref{lem:bi} shows that the proof will be completed once we have shown that
\be \label{surjective}
U \rightarrow L_2(I;H^{\frac12}(\partial\Omega)) \cap H^{\frac14}(I;L_2(\partial\Omega))\colon {\bf u} \mapsto u_D:=u_1|_{I \times \partial\Omega} \text{ is surjective.}
\ee
As shown in \cite[Thm.~2.9]{45.491}, the mapping
\begin{align*}
&u_1 \mapsto (h,u_D):=\Big(v \mapsto  \int_{I \times \Omega} \partial_t u_1 \,v +\nabla_{\bf x} u_1 \cdot \nabla_{\bf x} v \,d{\bf x}\,dt, u_1|_{I \times \partial \Omega}\Big) \\
&\in\Lis\Big(L_2(I;H^1(\Omega)) \cap H^{\frac12}_{00,\{0\}}(I;L_2(\Omega)), \\
&\quad\big(L_2(I;H_0^1(\Omega)) \cap H^{\frac12}_{00,\{T\}}(I;L_2(\Omega))\big)' \times L_2(I;H^{\frac12}(\partial\Omega)) \cap H^{\frac14}(I;L_2(\partial\Omega))\Big),
\end{align*}
where, with $H^1_{0,\{0\}}(I):= \{w \in H^1(I)\colon w(0)=0\}$, 
$H^{\frac12}_{00,\{0\}}(I):=[L_2(I),H^1_{0,\{0\}}(I)]_{\frac12}$, with a similar definition of $H^{\frac12}_{00,\{T\}}(I)$.
For given $h$ and  $u_D$, the corresponding $u_1$ is in $L_2(I;H^1(\Omega))$.
Taking $h \in L_2(I;L_2(\Omega))$ (e.g., $h=0$) and ${\bf u}_2 =-\nabla_{\bf x} u_1 \in L_2(I\times \Omega)^d$, 
from $\int_{I \times \Omega} \partial_t u_1 \,v - {\bf u}_2 \cdot \nabla_{\bf x} v \,d{\bf x}\,dt =\int_{I \times \Omega} h v \,d{\bf x}\,dt$ for
$v\in {\mathcal D}(I\times\Omega)\subset L_2(I;H_0^1(\Omega)) \cap H^{\frac12}_{00,\{T\}}(I;L_2(\Omega))$,
 it follows that 
$\divv {\bf u} = h\in L_2(I \times \Omega)$, i.e., \eqref{surjective} is valid.
\end{proof}

Using Theorem~\ref{thmDir}, we formulate the parabolic problem with inhomogeneous pure Dirichlet boundary conditions as a well-posed first-order system.
Let $(g_1,{\bf g}_2) \in L_2(I;L_2(\Omega))\times L_2(I;L_2(\Omega)^d)$, $u_0 \in L_2(\Omega)$, and $u_D \in L_2(I;H^{\frac12}(\partial\Omega))\cap H^{\frac14}(I;L_2(\partial\Omega))$, and set $g:= v \mapsto\int_{I \times \Omega} g_1 v +{\bf g}_2 \cdot \nabla_{\bf x} v \,d{\bf x}\,dt \in L_2(I;H^{-1}(\Omega))$.
Then the solution ${\bf u}=(u_1,{\bf u}_2) \in U$ of 
\begin{align}\label{eq:G_D_equation}
G_D{\bf u}=({\bf g}_2,g_1-{\bf b}\cdot {\bf A}^{-1}{\bf g}_2,u_0,u_D),
\end{align}
satisfies 
\begin{align*}
(Bu_1,\gamma_0 u_1,{u_1}|_{I \times \partial\Omega})=(g,u_0,u_D),
\end{align*}
i.e., $u_1$ satisfies  the parabolic PDE in standard variational form and both the initial and Dirichlet boundary condition.
 Indeed, knowing ${\bf u}_2+{\bf A} \nabla_{\bf x} u_1={\bf g}_2$, the second 
 equation in \eqref{eq:G_D_equation} is equivalent to 
 $\int_{I \times \Omega} (\partial_t u_1+{\bf b}\cdot \nabla_{\bf x} u_1 +c u_1)v+{\bf A} \nabla_{\bf x} u_1\cdot \nabla_{\bf x}  v \,d{\bf x} \,d t=g(v)$ for all $v \in L_2(I;H^1_0(\Omega))$.
 \medskip
  

Analogously to the case of inhomogeneous pure Dirichlet boundary conditions, the
combination of Theorem~\ref{thm2} and Lemma~\ref{lem:bi} allows to prove the following theorem for inhomogeneous pure Neumann boundary conditions.

\begin{theorem}[Inhomogeneous (pure) Neumann] \label{thmNeu} 
It holds that
\begin{align*}
&G_N\colon {\bf u} =(u_1,{\bf u}_2) \mapsto ({\bf u}_2+ {\bf A}\nabla_{\bf x} u_1, \divv {\bf u} -{\bf b}\cdot{\bf A}^{-1}{\bf u}_2 + c u_1,u_1(0,\cdot), {\bf u}|_{I \times \partial\Omega}\cdot {\bf n})
\\
 &\in \Lis\Big(U,L_2(I\times\Omega)^d \times L_2(I\times\Omega) \times L_2(\Omega) \times 
 \big(L_2(I;H^{\frac12}(\partial\Omega))\cap H^{\frac14}(I;L_2(\partial\Omega))\big)'
 \Big).
\end{align*}
\end{theorem}

\begin{proof}
Clearly, the first two components of $G_N$ are continuous. 
Recall from \eqref{eq:boundu1}--\eqref{eq:Hhalf} that also the third one is bounded, and that $\|u_1\|_{H^{\frac12}(I;L_2(\Omega))} \lesssim \|{\bf u}\|_U$.
To see boundedness of the fourth one, we first remark that for smooth ${\bf u}$ and $v$ on $I \times \Omega$, integration by parts shows that
\be \label{3}
\int_{I \times \partial\Omega} {\bf u} \cdot {\bf n} \,v \,d {\bf s} =
\int_{I \times \Omega} {\bf u}_2 \cdot \nabla_{\bf x} v + \divv {\bf u} \, v-
\partial_t u_1  v \,d {\bf x}\,d t.
\ee
As we have seen in the proof of Theorem~\ref{thmDir}, $v \in L_2(I;H^{\frac12}(\partial\Omega))\cap H^{\frac14}(I;L_2(\partial\Omega))$ has a bounded extension
to a $v_1 \in L_2(I;H^1(\Omega)) \cap H_{00,\{0\}}^{\frac12}(I;L_2(\Omega))$. Equally well it has a bounded
extension
to a $v_2 \in L_2(I;H^1(\Omega)) \cap H_{00,\{T\}}^{\frac12}(I;L_2(\Omega))$.
Taking  a smooth $\chi\colon I \rightarrow [0,1]$ with $\chi \equiv 1$ in a neighborhood of $0$ and $\chi \equiv 0$ in a neighborhood of $T$, and $v_3(t,x):=\chi(t) v_1(t,x) + (1-\chi(t)) v_2(t,x)$, we obtain
a bounded extension
to a $v_3 \in L_2(I;H^1(\Omega)) \cap H_{00}^{\frac12}(I;L_2(\Omega))$, where $H_{00}^{\frac12}(I):=[L_2(I),H^1_0(I)]_{\frac12}$.
Given such an extension of $v \in L_2(I;H^{\frac12}(\partial\Omega))\cap H^{\frac14}(I;L_2(\partial\Omega))$, for ${\bf u} \in U$ the right-hand side of \eqref{3} can be bounded by
a multiple of 
$\|{\bf u}\|_{U}$ $\|v\|_{L_2(I;H^{\frac12}(\partial\Omega))\cap H^{\frac14}(I;L_2(\partial\Omega))}$, 
where the term $\int_{I \times \Omega} \partial_t u_1 v_3 d{\bf x}\,dt$ is bounded via interpolation as follows
\begin{align*}
&\Big| \int_{I\times \Omega} \partial_t u_1 v_3\, d{\bf x}\,dt \Big|
\lesssim \|u_1\|_{[L_2(I;L_2(\Omega)),H^1(I;L_2(\Omega))]_{\frac12}}  \|v_3\|_{[H_0^1(I;L_2(\Omega)),L_2(I;L_2(\Omega))]_{\frac12} }
\\
&\quad\eqsim \|u_1\|_{H^{\frac12}(I;L_2(\Omega))} \|v_3\|_{H_{00}^{\frac12}(I;L_2(\Omega))}
\lesssim \|{\bf u}\|_{U} \,\|v\|_{L_2(I;H^{\frac12}(\partial\Omega))\cap H^{\frac14}(I;L_2(\partial\Omega))}.
\end{align*}
By a standard mollification argument as in the original proof of Meyers--Serrin, one sees that the set of smooth ${\bf u}\in U$ is dense in $U$.
This yields that $G_N$ is bounded. 

Since in the current case of $\Gamma_N=\partial\Omega$, 
we have $\{{\bf u} \in U \colon {\bf u}|_{I \times \partial\Omega}\cdot {\bf n}=0\}=U_0$, 
knowing the result of Theorem~\ref{thm2}, Lemma~\ref{lem:bi} shows that the proof will be completed once we have shown that
\be \label{surjective2}
U \rightarrow \big(L_2(I;H^{\frac12}(\partial\Omega))\cap H^{\frac14}(I;L_2(\partial\Omega))\big)'
\colon {\bf u} \mapsto {\bf u}|_{I \times \partial\Omega} \cdot {\bf n} \text{ is surjective.}
\ee
In \cite[Corollary~3.17]{45.491}, it has been shown that for any $\psi \in \big(L_2(I;H^{\frac12}(\partial\Omega))\cap H^{\frac14}(I;L_2(\partial\Omega))\big)'$ there exists a 
$u_1 \in L_2(I;H^1(\Omega)) \cap H^{\frac12}_{00,\{0\}}(I;L_2(\Omega))$ with
$\partial_t u_1 - \Delta_{\bf x} u_1=0$ on $I \times \Omega$, and $(\nabla_{\bf x} u_1)|_{I \times \partial\Omega} \cdot {\bf n}_{\bf x}=-\psi$.
Taking ${\bf u}_2=-\nabla_{\bf x} u_1$, it means $\divv {\bf u}=0$ and ${\bf u}|_{I \times \partial\Omega}\cdot {\bf n}=\psi$, so that ${\bf u} \in U$ and \eqref{surjective2} is valid.
\end{proof}

Using Theorem~\ref{thmNeu}, we formulate the parabolic problem with inhomogeneous pure Neumann boundary conditions as a well-posed first-order system.
Let $(g_1,{\bf g}_2) \in L_2(I;L_2(\Omega))\times L_2(I;L_2(\Omega)^d)$ with ${\bf g}_2|_{I \times \partial\Omega} \cdot {\bf n}_{\bf x} \in  \big(L_2(I;H^{\frac12}(\partial\Omega))\cap H^{\frac14}(I;L_2(\partial\Omega))\big)'$, $u_0 \in L_2(\Omega)$, and $\phi \in (L_2(I;H^{\frac12}(\partial\Omega))\cap H^{\frac14}(I;L_2(\partial\Omega)))'$, and set  $g:= v \mapsto\int_{I \times \Omega} g_1 v +{\bf g}_2 \cdot \nabla_{\bf x} v \,d{\bf x}\,dt \in L_2(I;H^{-1}(\Omega))$.
Then, the solution ${\bf u}=(u_1,{\bf u}_2)$ $\in U$ of 
\begin{align}\label{eq:G_N_equation}
G_N{\bf u}=({\bf g}_2,g_1-{\bf b}\cdot {\bf A}^{-1}{\bf g}_2,u_0,{\bf g}_2|_{I \times \partial\Omega} \cdot {\bf n}_{\bf x} -\phi),
\end{align}
satisfies 
\begin{align*}
(Bu_1,\gamma_0 u_1,{{\bf A} \nabla_{\bf x} u_1|_{I \times \partial\Omega}} \cdot {\bf n}_{\bf x})=(g,u_0,\phi),
\end{align*}
i.e., $u_1$ satisfies the parabolic PDE in standard variational form and both the initial and Neumann boundary condition.
Indeed, knowing ${\bf u}_2+{\bf A} \nabla_{\bf x} u_1={\bf g}_2$, 
it holds that ${{\bf A} \nabla_{\bf x} u_1|_{I \times \partial\Omega}} \cdot {\bf n}_{\bf x}=({\bf g}_2-{\bf u}_2)|_{I \times \partial\Omega} \cdot {\bf n}_{\bf x}=\phi$, and 
the second 
equation in \eqref{eq:G_N_equation} is equivalent to 
$\int_{I \times \Omega} (\partial_t u_1+{\bf b}\cdot \nabla_{\bf x} u_1 +c u_1)v+{\bf A} \nabla_{\bf x} u_1\cdot \nabla_{\bf x}  v \,d{\bf x} \,d t=g(v)$ for all $v \in L_2(I;H^1_0(\Omega))$.

\newpage

\section{Plain convergence of adaptive algorithm for\\ homogeneous pure Dirichlet boundary conditions}

Consider the setting of Subsection~\ref{S2ndorder} with $\Gamma_D=\partial\Omega$ and homogeneous Dirichlet datum $u_D=0$.
Let $(f_1,{\bf f}_2) \in L_2(I;L_2(\Omega))\times L_2(I;L_2(\Omega)^d)$, $f:=v\mapsto\int_{I \times \Omega} f_1 v$ $+{\bf f}_2 \cdot \nabla_{\bf x} v \,d{\bf x}\,dt$ $\in L_2(I;H^{-1}(\Omega))$ and $u_0\in L_2(\Omega)$.
Since no Neumann boundary conditions are present, $g$ from \eqref{eq:RHS g} coincides with $f$. 
Then, with $u$ being the solution $u$ of \eqref{2nd}, Proposition~\ref{prop:Bu2least} states that ${\bf u}=(u,-{\bf A}\nabla_{\bf x} u)$ $\in U_0$  is the unique solution of
\begin{align*}
G{\bf u} = {\bf f},
\end{align*}
where
\begin{align*}
{\bf f}:=({\bf f}_2,  f_1 -{\bf b}\cdot {\bf A}^{-1} {\bf f}_2, u_0)\in L:=L_2(I\times\Omega)^{d}\times L_2(I\times\Omega)\times L_2(\Omega).
\end{align*}%

For an arbitrary discrete subspace $U_0^\delta \subset U_0$, the corresponding least-squares approximation ${\bf u}^\delta\in U_{0}^\delta$ of ${\bf u}$ is given by
\begin{align}\label{eq:argmin}
{\bf u}^\delta:=\argmin\limits_{{\bf v} \in U_{0}^\delta} \|{\bf f} - G{\bf v}\|_{L}^2. 
\end{align}
The resulting Euler--Lagrange equation reads as
\begin{align}\label{eq:Galerkin}
 \langle G {\bf u}^\delta,G {\bf v}\rangle_{L}=\langle {\bf f},G {\bf v}\rangle_{L} \quad \text{for all } {\bf v}\in U_0^\delta.
\end{align}
As $G$ is a linear isomorphism, the left-hand side defines an elliptic bilinear form and the Lax--Milgram lemma indeed guarantees unique solvability of \eqref{eq:argmin}--\eqref{eq:Galerkin}.

Throughout the remainder of this section, for $p \in \N$ some fixed polynomial degree  we consider discrete spaces of the form
\begin{align*}
 U_{0}^\delta :=
 S^p_0(\TT^\delta) \times S^p(\TT^\delta)^d
 \subset U_{0}
\end{align*}
for  conforming simplicial meshes $\TT^\delta$ of $I\times\Omega$, 
where 
\begin{align*}
S^p(\TT^\delta) &:= \{u \in C(I\times\Omega) :\, u|_K \text{ polynomial of degree }p \text{ for all }K\in\TT^\delta\},
\\
S^p_0(\TT^\delta) &:= \{u\in S^p(\TT^\delta) :\, u|_{I\times\partial\Omega}=0 \}.
\end{align*}
In particular, we consider such  meshes that can be created by newest vertex bisection (\cite{249.87}) starting from a given initial partition 
$\TT^0$.


Finally, we define the reliable and efficient {\sl a posteriori} error estimator 
\begin{align}\label{eq:estimator}
 \eta({\bf f},{\bf u}^\delta):=  
 \| {\bf f} - G {\bf u}^\delta \|_{L} \eqsim \| {\bf u} - {\bf u}^\delta\|_U,
\end{align}
with corresponding error indicators
\begin{align}\label{eq:indicators}
 \eta(K; {\bf f}, {\bf u}^\delta):=  \| {\bf f} - G {\bf u}^\delta \|_{L(K)} 
 \quad\text{for all } K\in\TT^\delta,
\end{align}
where
\begin{align*}
 L(\omega) := L_2(\omega)^{d}\times L_2(\omega) \times L_2(\partial_0 \omega)
 \quad\text{for all measurable }\omega\subseteq I\times\Omega.
\end{align*}
Here and throughout the remainder of this section, we use the notation $\partial_0\omega:=\partial\omega \cap (\{0\}\times \Omega)$.

We consider the following adaptive algorithm.
\begin{algorithm}
\label{alg:adaptive}
\textbf{Input:} 
Right-hand side ${\bf f}\in L$, initial mesh $\TT^0=\TT^{\delta_0}$, marking function $M:[0,\infty)\to[0,\infty)$ that is continuous at $0$ with $M(0)=0$.
\\
\textbf{Loop:} For each $\ell=0,1,2,\dots$, iterate the following steps {\rm(i)--(iv)}:
\begin{enumerate}[\rm(i)]
\item  Compute least-squares approximation ${\bf u}^\ell={\bf u}^{\delta_\ell}$ of ${\bf u}$. %
\item Compute  error indicators $\eta({K;{\bf f}, {\bf u}^\ell})$ for all elements ${K}\in\TT^\ell=\tria^{\delta_\ell}$. %
\item\label{item:marking}  Determine a set of marked elements $\mathcal{M}^\ell\subseteq\TT^\ell$ with the following marking property
\begin{align*}
 \max_{K\in\TT^\ell\setminus\mathcal{M}^\ell} \eta(K;{\bf f}, {\bf u}^\ell) \le M(\max_{K\in\mathcal{M}^\ell} \eta(K;{\bf f}, {\bf u}^\ell)).
\end{align*}
\item Generate refined conforming simplicial mesh $\TT^{\ell+1}$ by refining at least all marked elements $\mathcal{M}^\ell$ via newest vertex bisection. 
\end{enumerate}
\textbf{Output:} Refined meshes $\TT^\ell$, corresponding exact discrete solutions ${\bf u}^\ell$, and 
error estimators $\eta({\bf f}, {\bf u}^\ell)$ for all $\ell \in \N_0$.
\end{algorithm}

\begin{remark}\label{rem:marking}
The criterion~{\rm (iii)} is satisfied for standard marking strategies:

\begin{itemize}
\item Suppose that the D\"orfler criterion is used for fixed $0<\theta\le1$ , i.e., 
\begin{align*}
\theta\,\eta({\bf f},{\bf u}^\ell)^2 \le \sum_{K\in \mathcal{M}^\ell} \eta(K; {\bf f}, {\bf u}^\ell)^2.
\end{align*}
While this does not directly imply {\rm (iii)}, with the aim to realize optimal rates, 
the set $\mathcal{M}^\ell$ is constructed in practice via sorting of the indicators such that also
\begin{align*}
 \max_{K\in\TT^\ell\setminus\mathcal{M}^\ell} \eta(K;{\bf f}, {\bf u}^\ell) 
 \le \min_{K\in\mathcal{M}^\ell} \eta(K;{\bf f}, {\bf u}^\ell),
\end{align*}
see \cite{pp19}.
Then, {\rm (iii)} holds with $M(t):=t$.

\item Suppose the maximum criterion is used for fixed $0\le \theta\le1$, i.e.,
\begin{align*}
 \mathcal{M}^\ell:=\{K\in\TT^\ell:\eta(K;{\bf f}, {\bf u}^\ell) \ge (1-\theta) \max_{K'\in\TT^\ell} \eta(K';{\bf f}, {\bf u}^\ell)\}.
\end{align*}
Then, {\rm (iii)} holds with $M(t):=t$.
To see this, let $K\in\TT^\ell\setminus\mathcal{M}^\ell$ and note that
\begin{align*}
 \eta(K;{\bf f}, {\bf u}^\ell) < (1-\theta) \max_{K'\in\TT^\ell} \eta(K';{\bf f}, {\bf u}^\ell) \le \min_{K'\in\mathcal{M}^\ell} \eta(K';{\bf f}, {\bf u}^\ell). 
 \end{align*}

\end{itemize}
\end{remark}

The following theorem states convergence of Algorithm~\ref{alg:adaptive}. 
For the heat equation with ${\bf A}=\identity$, ${\bf b}=0$, and $c=0$, the performance of the algorithm has been numerically investigated in \cite{75.257}.

\begin{theorem}[Convergence for homogeneous (pure) Dirichlet]\label{thm:error_convergence}
There holds   plain convergence of the error 
\begin{align}
 \| {\bf u} - {\bf u}^\ell \|_U \to 0 \quad \text{ as } \ell\to \infty.
\end{align}
As the considered estimator~\eqref{eq:estimator} is equivalent to the error,  convergence to zero also transfers to the estimator.
\end{theorem}

\begin{proof}
It suffices to verify that the considered problem fits into the abstract framework of \cite{249.025}, which gives sufficient conditions for error convergence. 
This will be done in the following three steps.

{\bf Step~1:}
Define another equivalent norm on $U$
\begin{align*}
 \| {\bf v} \|_{U(I\times\Omega)}^2 := \|v_1\|_{L_2(I;H^1(\Omega))}^2+\|{\bf v}_2\|_{L_2(I;L_2(\Omega)^d)}^2+\|\divv {\bf v}\|_{L_2(I\times\Omega)}^2 + \| v_1(0,\cdot) \|_{L_2(\Omega)}^2
\end{align*}
for all ${\bf v}=(v_1,{\bf v}_2)\in U$.
Moreover, define the following semi-norms for all measurable subsets $\omega\subseteq I\times\Omega$
\begin{align*}
 \| {\bf v} \|_{U(\omega)}^2 := \|v_1\|_{L_2(\omega)}^2 +  \|\nabla_{\bf x} v_1\|_{L_2(\omega)}^2 +\|{\bf v}_2\|_{L_2(\omega)}^2+\|\divv {\bf v}\|_{L_2(\omega)}^2 + \| v_1|_{\partial_0\omega} \|_{L_2(\partial_0\omega)}^2. 
\end{align*}
The additional term $\| v_1|_{\partial_0\omega} \|_{L_2(\partial_0\omega)}^2$ will  be required to prove local stability~\eqref{eq:G locally stable}.
The semi-norms are \emph{additive} as well as \emph{absolutely continuous} in the sense of \cite[Section~2.1]{249.025}, i.e.,
\begin{align}\label{eq:subadditive}
\begin{split}
 \| {\bf v} \|_{U(\omega_1\cup\omega_2)}^2 =  \| {\bf v} \|_{U(\omega_1)}^2 +  \| {\bf v} \|_{U(\omega_2)}^2
 \quad\text{for all } {\bf v} \in U, 
 \omega_1,\omega_2\subseteq {I\times \Omega}\\
 \text{ with }\omega_1\cap\omega_2=\emptyset;
 \end{split}
\end{align}
as well as 
\begin{align}\label{eq:absolutely_continuous}
 \lim_{|\omega|\to0}\| {\bf v} \|_{U(\omega)}^2 = 0\quad\text{for all }{\bf v}\in U.
\end{align}

\begin{remark}
For this proof step it was essential that we got rid of the dual norm in Proposition~\ref{equiv}, see also Remark~\ref{litt}.
\end{remark}

{\bf Step~2:} 
We next show a \emph{local approximation property} in the sense of \cite[Section~2.2.2]{249.025}, i.e., 
existence of a dense subspace $W\subseteq U_0$ equipped with additive semi-norms $\| \cdot \|_{W(\omega)}$, $\omega\subseteq I\times\Omega$, such that $\| \cdot \|_{W({I\times\Omega})}=\| \cdot \|_{W}$, and a corresponding $\Pi^\delta \in \cL(W,U_0^\delta)$
with 
\begin{align}\label{eq:local_approximation}
 \| {\bf v} - \Pi^\delta {\bf v} \|_{U(K)} \lesssim |K|^{\frac{q}{d+1}} \| {\bf v} \|_{W(K)}
 \quad\text{for all }{\bf v}\in W, K\in\TT^\delta,
\end{align}
where $q>0$ is some fixed exponent. For  $k:=\min\{k' \in \N\colon k' \geq p+1,k' >\frac{d+1}{2}\}$, let 
\begin{align*}
 W:=\{{\bf v}=(v_1,{\bf v}_2) \in H^k(I\times\Omega)\times H^k(I\times\Omega)^{d}:\, v_1|_{I\times\partial\Omega}=0 \}\subset U_0,
\end{align*}
and let
$I^\delta \in \cL(H^k (I\times\Omega),S^p(\TT^\delta))$ be the standard point-wise interpolation operator, which is well-defined because of $k >\frac{d+1}{2}$.
Then, the operator $\Pi^\delta:={\bf I}^\delta_{d+1}:=(I^\delta,\dots, I^\delta)$ (of length $d+1$) is in $\cL(W,U_0^\delta)$, and with ${\bf I}^\delta_d$ defined analogously, it holds that
\begin{align*}
  &\| {\bf v} - \Pi^\delta {\bf v} \|_{U(K)}^2 
  = \|v_1 - I^\delta v_1\|_{L_2(K)}^2 
  +  \|\nabla_{\bf x} (v_1-I^\delta v_1)\|_{L_2(K)}^2 
  + \| (v_1-I^\delta v_1)|_{\partial_0 K}\|_{L_2(\partial_0 K)}^2
  \\
  &\qquad+\|{\bf v}_2- {\bf I}^\delta_{d}{\bf v}_2\|_{L_2(K)}^2
  +\|\divv ({\bf v}- {\bf I}^\delta_{d+1}{\bf v})\|_{L_2(K)}^2 
  \\
  &\quad\lesssim \|(v_1-I^\delta v_1)|_{\partial_0 K}\|_{L_2(\partial_0 K)}^2
  + \|v_1 - I^\delta v_1\|_{H^1(K)}^2 
  + \|{\bf v}_2- {\bf I}^\delta_{d}{\bf v}_2\|_{H^1(K)}^2.
\end{align*}
A standard trace inequality~\cite[Equation~(10.3.8)]{35.7} further shows that 
\begin{align*}
 &\| (v_1-I^\delta v_1)|_{\partial_0 K}\|_{L_2(\partial_0 K)}^2 
 \le 
 \| (v_1-I^\delta v_1)|_{\partial K}\|_{L_2(\partial K)}^2
 \\
 &\quad\lesssim |K|^{-\frac{1}{d+1}} \|v_1 - I^\delta v_1\|_{L_2(K)}^2 
 + |K|^{\frac{1}{d+1}} |v_1 - I^\delta v_1|_{H^1(K)}^2. 
\end{align*}

To finish the proof, we show for $m \in \{0,1\}$ and $v\in H^k(I\times\Omega)$ that
\begin{align*}
 \|v -I^\delta v \|_{H^m(K)} \lesssim |K|^{\frac{p+1-m}{d+1}} \|v\|_{H^k(K)}.
\end{align*}
While this is standard if $p+1>(d+1)/2$, i.e., $k=p+1$, it is not evident if $p+1\le(d+1)/2$, and we thus provide a short proof.
We first assume that $K$ is the reference simplex, i.e., the convex hull of the canonical basis vectors in $\R^{d+1}$.
Let $\widetilde v\in P^{k-1}$ be the best approximation of $v$ with respect to $\| \cdot \|_{H^k}$ in the space of polynomials of degree $k-1$, 
and let $\widehat v\in P^{p}$ be the best approximation of $\widetilde v$ with respect to $\| \cdot \|_{H^k}$ in the space of polynomials of degree $p$.
The projection property as well as continuity of $I^\delta$ on $H^k(K)$ show that 
\begin{align*}
 \|v -I^\delta v \|_{H^m(K)} &= \|(\identity -I^\delta) (v- \widehat v) \|_{H^m(K)}
 \\
 &\lesssim \| v- \widehat v \|_{H^k(K)}
 \le \| v - \widetilde v \|_{H^k(K)} + \| \widetilde v - \widehat v \|_{H^k(K)}.
\end{align*}
Equivalence of norms on finite-dimensional spaces and two applications of the Bramble--Hilbert lemma further yield that
\begin{align*}
 &\| v - \widetilde v \|_{H^k(K)} + \| \widetilde v - \widehat v \|_{H^k(K)}
 \eqsim \| v - \widetilde v \|_{H^k(K)} + \| \widetilde v - \widehat v \|_{H^{p+1}(K)}
 \\
 &\qquad\lesssim \| v - \widetilde v \|_{H^k(K)} + | \widetilde v |_{H^{p+1}(K)}
 \le \| v - \widetilde v \|_{H^k(K)} + | v - \widetilde v |_{H^{p+1}(K)} + | v |_{H^{p+1}(K)}
 \\
 &\qquad \le 2\| v - \widetilde v \|_{H^k(K)} + | v |_{H^{p+1}(K)} \lesssim | v |_{H^k(K)} + | v |_{H^{p+1}(K)}.
\end{align*}
If $K$ is arbitrary, the fact that we use newest vertex bisection allows to apply a standard scaling argument, which yields that
\begin{align*}
 \|v -I^\delta v \|_{H^m(K)} 
 \lesssim  |K|^{\frac{k-m}{d+1}} |v|_{H^k(K)} + |K|^{\frac{p+1-m}{d+1}} |v|_{H^{p+1}(K)}
 \lesssim
 |K|^{\frac{p+1-m}{d+1}} \| v \|_{H^k(K)}.
\end{align*}
Overall, we thus conclude \eqref{eq:local_approximation} with $q=p$.

{\bf Step~3:}
With the patch $\omega^\delta(K):=\bigcup\{K'\in\TT^\delta:\, K\cap K'\neq\emptyset \}$ of an element $K\in\TT^\delta$, we finally show that the employed error estimator is \emph{locally stable} as in \cite[Section~2.2.3]{249.025}, i.e.,
\begin{align}\label{eq:local_stability}
 \eta(K;{\bf f}, {\bf u}^\delta) \lesssim \|{\bf u}^\delta \|_{U(\omega^\delta(K))} + \| D \|_{\widetilde W(\omega^\delta(K))} 	
 \quad\text{for all }K\in\TT^\delta
\end{align}
for a suitable $D$ depending only on the data in a normed space $\widetilde W$ equipped with additive and absolutely continuous semi-norms $\| \cdot \|_{\widetilde W(\omega)}$, $\omega\subseteq I\times\Omega$, such that $\| \cdot \|_{\widetilde W(I\times\Omega)}=\| \cdot \|_{\widetilde W}$; 
as well as \emph{strongly reliable} as in \cite[Section~2.2.3]{249.025}
\begin{align}\label{eq:strong_reliability}
\begin{split}
 \langle {\bf f} - G{\bf u}^\delta,  G {\bf v} \rangle_{L}
 \lesssim \sum_{K\in\TT^\delta} \eta(K;{\bf f}, {\bf u}^\delta) \| {\bf v} \|_{U(\omega^\delta(K))}
 \quad\text{for all }{\bf v}\in U_0.
 \end{split}
\end{align}
\begin{remark}Actually, \cite{249.025} assumes that $\widetilde W=L_2(I\times\Omega)$.
It is, however, straightforward to see that our mildly relaxed assumption is already sufficient for the convergence proof.
Indeed, local stability is only employed in the elementary \cite[Lemma~3.5]{249.025}.
\end{remark}
Local stability~\eqref{eq:local_stability} follows from the triangle inequality
\begin{align*}
 \eta(K) = \eta(K;{\bf f}, {\bf u}^\delta) = \| {\bf f} - G {\bf u}^\delta \|_{L(K)}
 \le \| {\bf f} \|_{L(K)} + \| G {\bf u}^\delta \|_{L(K)}
\end{align*}
and the following local stability of $G$
\begin{align}\label{eq:G locally stable}
\begin{split}
  \| G {\bf u}^\delta \|_{L(K)}^2
  &\lesssim \| {\bf u}_2^\delta \|_{L_2(K)}^2 
 + \| \nabla_{\bf x} u_1^\delta \|_{L_2(K)}^2 
 + \|\divv {\bf u}^\delta \|_{L_2(K)}^2   
 + \|  u_1^\delta \|_{L_2(K)}^2 
 \\
 &\quad+ \| u_1^\delta(0,\cdot) \|_{L_2(\partial_0 K)}^2
 =  \| {\bf u}^\delta \|_{U(K)}^2.
 \end{split}
\end{align}
Strong reliability~\eqref{eq:strong_reliability} follows from the Cauchy--Schwarz inequality together with the previous local stability of $G$
\begin{align*}
  \langle {\bf f} - G{\bf u}^\delta,  G {\bf v} \rangle_{L}
  \le \sum_{K\in\TT^\delta} \eta(K;{\bf f},{\bf u}^\delta) \, \| G {\bf v}\|_{L(K)} 
  \lesssim \sum_{K\in\TT^\delta} \eta(K;{\bf f}, {\bf u}^\delta) \| {\bf v} \|_{U(K)},
\end{align*}
which concludes the proof.
\end{proof}

\begin{remark}
Together with the C\'ea lemma and with $h^\delta_{\rm max}:=\max\{|K|^{1/(d+1)}: K\in\TT^\delta\}$,  Step~2 from the proof particularly yields the {\sl a priori} estimate
\begin{align*}
 \|{\bf u} - {\bf u}^\delta\|_U \lesssim \inf_{{\bf v}\in U_{0}^\delta} \| {\bf u} - {\bf v}\|_U 
 \le   \| {\bf u} - \Pi^\delta{\bf u}\|_U 
 \lesssim  (h_{\rm max}^\delta)^p \| {\bf u} \|_{H^k(I\times\Omega)\times H^k(I\times\Omega)^d}
\end{align*}
whenever the solution ${\bf u}$ satisfies the additional regularity ${\bf u} \in H^k(I\times\Omega)\times H^k(I\times\Omega)^{d}$, 
where $k=\min\{k' \in \N\colon k' \geq p+1,\,k' >\frac{d+1}{2}\}$.
Instead of the standard interpolation operator $I^\delta$, one can also consider the Scott--Zhang operator $\widetilde I^\delta$ from \cite{247.2} which preserves homogeneous Dirichlet boundary conditions.
Then, \cite[Equation~(4.3)]{247.2} gives an alternative local bound for the resulting operator $\widetilde \Pi^\delta$
\begin{align*}
 \| {\bf v} - \widetilde\Pi^\delta {\bf v} \|_{U(K)} \lesssim  |K|^{\frac{p}{d+1}} \| {\bf v} \|_{H^{p+1}(\omega^\delta(K))\times H^{p+1}(\omega^\delta(K))^{d}} 
\end{align*}
for all ${\bf v} \in H^{p+1}(I\times\Omega)\times H^{p+1}(I\times\Omega)^{d}$ with ${\bf v}|_{I\times\partial\Omega} = 0$ and all $K\in\TT^\delta$.
In particular this yields the {\sl a priori} estimate
\begin{align}
 \|{\bf u} - {\bf u}^\delta\|_U 
 \lesssim  (h_{\rm max}^\delta)^p \| {\bf u} \|_{H^{p+1}(I\times\Omega)\times H^{p+1}(I\times\Omega)^d}
\end{align}
under the milder assumption that ${\bf u} \in H^{p+1}(I\times\Omega)\times H^{p+1}(I\times\Omega)^{d}$.
We mention that \cite[Theorem~14]{75.257} already proved the latter inequality in the lowest-order case $p=1$ under even weaker assumptions on ${\bf u}$.
However, their proof is restricted to simplicial meshes that directly result from a tensor-product mesh~\cite[Section~4.1.2]{75.257}.
\end{remark}

\begin{remark}\label{rem:general ls}
{\rm (a)} We stress that the proof of Theorem~\ref{thm:error_convergence} is relatively abstract in the sense that it  generalizes to a large class of least-squares formulations:
Suppose that $U$ (instead of $U_0$) and $L$ are arbitrary Hilbert spaces.
Consider the equation
\begin{align*}
 G u = f \quad\text{for given } G\in\Lis(U,L) \text{ and } f\in L.
\end{align*}
Moreover, suppose that $U$ as well as $L$ are equipped with additive and absolutely continuous (see \eqref{eq:subadditive}--\eqref{eq:absolutely_continuous}) semi-norms $\|\cdot\|_{U(\omega)}$, $\|\cdot\|_{L(\omega)}$ for all measurable subsets $\omega$ of some set $\Omega\subseteq\R^n$ 
being the union of an initial conforming simplicial mesh $\TT^0$.
To any conforming simplicial mesh $\TT^\delta$ of $\Omega$, we associate a finite-dimensional subspace $U^\delta\subseteq U$ such that $U^\delta\subseteq U^{\widehat\delta}$ for all refinements $\TT^{\widehat\delta}$ of $\TT^\delta$.
We define the least-squares approximation $u^\delta$ as in \eqref{eq:argmin}--\eqref{eq:Galerkin} 
and the error estimator $\eta(f,u^\delta)$ with indicators $\eta(K;f,u^\delta)$ as in \eqref{eq:estimator}--\eqref{eq:indicators}.
In this setting, Algorithm~\ref{alg:adaptive} can be applied. 
Then, the (analogous) local approximation property of Step~2 (where one could also allow for $W((\omega^\delta)^m(K))$ for fixed $m\in\N$ instead of $W(K)$ in~\eqref{eq:local_approximation}) and local stability of $G$ as in~\eqref{eq:G locally stable} (where again $U(K)$ could be replaced by $U((\omega^\delta)^{m}(K))$) yield error and estimator convergence
\begin{align}\label{eq:abstract convergence}
 \|u - u^\ell\|_{U} \eqsim \eta(f,u^\ell) \to 0 \quad \text{as }\ell\to\infty.
\end{align}
Independently, it has also been recently observed in \cite{fp20} that the given abstract assumptions yield \eqref{eq:abstract convergence} for least-squares methods. 
However, we stress that Theorem~\ref{thm:error_convergence} is not available in \cite{fp20}. 

{\rm (b)}
The setting of {\rm (a)} is for instance satisfied for a  standard least-squares formulation of the Poisson model problem~\cite[page~56]{23.5}, 
the Helmholtz problem \cite{35.93005}, the linear elasticity problem \cite{cks05}, and the Stokes problem \cite{clw04}, 
see also \cite[Chapter~3]{storn19} for a brief overview of all these formulations.
The involved spaces $H^1(\Omega)$ and $H({\rm div};\Omega)$ 
can be discretized by usual  finite element spaces, i.e., continuous piecewise polynomials and Raviart--Thomas functions, 
respectively.
The required  corresponding approximation properties~\eqref{eq:local_approximation} are well-known,
 see, e.g., \cite[Section~1.5]{eg04}. 
 

Only for the Stokes problem~ \cite{clw04}, 
one requires a special interpolation operator on (a dense subspace of) $\{{\bf v}\in H({\rm div};\Omega)^d: \int_\Omega {\rm tr}({\bf v})\,d{\bf x}=0 \}$, where ${\rm tr}$ denotes the trace of square matrices.
Since $S^1(\TT^\delta)^d$ is contained in the  Raviart--Thomas space of order $\ge1$ (excluding the lowest-order case), 
such an operator can be defined component-wise as an integral-preserving  $J^\delta$ $\in$  $\cL(H^2(\Omega),S^1(\TT^\delta))$ with a local approximation property, i.e.,  $\int_\Omega v \,d {\bf x} = \int_\Omega J^\delta v \,d {\bf x}$ and 
\begin{align}\label{eq:new interpolation}
\|v -J^\delta v \|_{H^1(K)} \lesssim |K|^{\frac{1}{d}} \|v\|_{H^2((\omega^\delta)^m(K))}
\end{align}
 for all $v\in H^2(\Omega)$, $K\in\TT^\delta$, and some fixed $m\in\N_0$.  The operator  $J^\delta$ is for instance constructed as follows: 
Inspired by \cite[Section~4.1]{249.97} and given the nodal Lagrange basis $\{\phi_i: i\in\{1,\dots,N\}\}$ with corresponding local dual basis $\{\psi_i: i\in\{1,\dots,N\}\}$ as in \cite{247.2}, one first defines  
\begin{align*}
\widetilde \psi_i:=\frac{\phi_i+\int_\Omega (1-\phi_i)\phi_i\,d{\bf x} \,\psi_i-\sum_{j \neq i}\big( \int_\Omega\phi_i \phi_j\,d{\bf x}\, \psi_j\big)}{\int_\Omega \phi_i\, d{\bf x}}
\end{align*}
for all $i\in\{1,\dots,N\}$. 
This provides a second local dual basis in the sense that $\supp(\widetilde\psi_i)\subset\supp(\phi_i)$ and $\int_{\supp(\widetilde\psi_j)} \phi_i\widetilde\psi_j\,d{\bf x} = \delta_{ij}$ for all $i,j\in\{1,\dots,N\}$.
Moreover, from $\sum_{i} \phi_i=\1$, one verifies that $\sum_i (\int_\Omega \phi_i d{\bf x}) \widetilde\psi_i=\1$ meaning that this dual basis has (lowest-order) approximation properties. 
Defining \begin{align*}J^\delta: H^1(\Omega)\to S^1(\TT^\delta),\quad v\mapsto \sum_{i=1}^N \int_{\supp \widetilde \psi_i} v \widetilde \psi_i\,d{\bf x} \,\phi_i,\end{align*}
the latter property implies that this biorthogonal projector
is integral-preserving, and the desired approximation property~\eqref{eq:new interpolation} with $m=2$ follows as in \cite{247.2}.

Moreover, \cite{fp20} verifies the setting of {\rm (a)}  for another least-squares formulation of the Stokes problem as well as the Maxwell problem.

{\rm (c)} Optimal convergence of adaptive least-square finite element methods driven by an equivalent weighted error estimator has been already proved 
for the Poisson problem 
in \cite{35.93557,carstensen20}, 
the linear elasticity problem~\cite{bcs18}, and the Stokes problem~\cite{bc17}.
However, apart from the very recent and independent work~\cite{fp20}, convergence for adaptive algorithms driven by the natural estimator is only known for the Poisson problem if  D\"orfler marking with a sufficiently large bulk parameter is used, see  \cite{cpb17}, where $Q$-linear convergence has been demonstrated. 
\end{remark}

\bibliographystyle{alpha}
\bibliography{../ref}

\begin{thebibliography}{CLMM94}

\bibitem[And13]{11}
R.~Andreev.
\newblock Stability of sparse space-time finite element discretizations of
  linear parabolic evolution equations.
\newblock {\em IMA J. Numer. Anal.}, 33(1):242--260, 2013.

\bibitem[BC17]{bc17}
P.~Bringmann and C.~Carstensen.
\newblock {h-adaptive least-squares finite element methods for the 2D Stokes
  equations of any order with optimal convergence rates}.
\newblock {\em Comput. Math. Appl.}, 74(8):1923--1939, 2017.

\bibitem[BCS18]{bcs18}
P.~Bringmann, C.~Carstensen, and G.~Starke.
\newblock {An adaptive least-squares FEM for linear elasticity with optimal
  convergence rates}.
\newblock {\em SIAM J. Numer. Anal.}, 56(1):428--447, 2018.

\bibitem[BG09]{23.5}
P.~B. Bochev and M.~D. Gunzburger.
\newblock {\em Least-squares finite element methods}, volume 166 of {\em
  Applied Mathematical Sciences}.
\newblock Springer, New York, 2009.

\bibitem[BJ89]{18.63}
I.~Babu{\v{s}}ka and T.~Janik.
\newblock The {$h$}-{$p$} version of the finite element method for parabolic
  equations. {I}. {T}he {$p$}-version in time.
\newblock {\em Numer. Methods Partial Differential Equations}, 5(4):363--399,
  1989.

\bibitem[BJ90]{18.64}
I.~Babu{\v{s}}ka and T.~Janik.
\newblock The {$h$}-{$p$} version of the finite element method for parabolic
  equations. {II}. {T}he {$h$}-{$p$} version in time.
\newblock {\em Numer. Methods Partial Differential Equations}, 6(4):343--369,
  1990.

\bibitem[BS08]{35.7}
S.~C. Brenner and L.~R. Scott.
\newblock {\em The mathematical theory of finite element methods}, volume~15 of
  {\em Texts in Applied Mathematics}.
\newblock Springer, New York, third edition, 2008.

\bibitem[Car20]{carstensen20}
Carsten Carstensen.
\newblock Collective marking for adaptive least-squares finite element methods
  with optimal rates.
\newblock {\em Math. Comp.}, 89(321):89--103, 2020.

\bibitem[CKS05]{cks05}
Z.~Cai, J.~Korsawe, and G.~Starke.
\newblock An adaptive least squares mixed finite element method for the
  stress-displacement formulation of linear elasticity.
\newblock {\em Numer. Methods Partial Differential Equations}, 21(1):132--148,
  2005.

\bibitem[CLMM94]{35.93005}
Z.~Cai, R.~Lazarov, T.~A. Manteuffel, and S.~F. McCormick.
\newblock First-order system least squares for second-order partial
  differential equations. {I}.
\newblock {\em SIAM J. Numer. Anal.}, 31(6):1785--1799, 1994.

\bibitem[CLW04]{clw04}
Z.~Cai, B.~Lee, and P.~Wang.
\newblock {Least-squares methods for incompressible Newtonian fluid flow:
  Linear stationary problems}.
\newblock {\em SIAM J. Numer. Anal.}, 42(2):843--859, 2004.

\bibitem[Cos90]{45.491}
M.~Costabel.
\newblock Boundary integral operators for the heat equation.
\newblock {\em Integral Equations Operator Theory}, 13(4):498--552, 1990.

\bibitem[CP15]{35.93557}
C.~Carstensen and E.-J. Park.
\newblock Convergence and optimality of adaptive least squares finite element
  methods.
\newblock {\em SIAM J. Numer. Anal.}, 53(1):43--62, 2015.

\bibitem[CPB17]{cpb17}
C.~Carstensen, E.-J. Park, and P.~Bringmann.
\newblock Convergence of natural adaptive least squares finite element methods.
\newblock {\em Numer. Math.}, 136(4):1097--1115, 2017.

\bibitem[DL92]{63}
R.~Dautray and J.-L. Lions.
\newblock {\em Mathematical analysis and numerical methods for science and
  technology. {V}ol. 5}.
\newblock Springer-Verlag, Berlin, 1992.
\newblock Evolution problems I.

\bibitem[DS18]{64.577}
D.~Devaud and Ch. Schwab.
\newblock Space-time {$hp$}-approximation of parabolic equations.
\newblock {\em Calcolo}, 55(3):Art. 35, 23, 2018.

\bibitem[EG04]{eg04}
A.~Ern and J.-L. Guermond.
\newblock {\em Theory and practice of finite elements}, volume 159 of {\em
  Applied Mathematical Sciences}.
\newblock Springer, New York, 2004.

\bibitem[FK19]{75.257}
T.~F\"uhrer and M.~Karkulik.
\newblock Space-time least-squares finite elements for parabolic equations,
  2019.
\newblock arXiv:1911.01942.

\bibitem[FP20]{fp20}
T.~F\"uhrer and D.~Praetorius.
\newblock A short note on plain convergence of adaptive least-squares finite
  element methods.
\newblock {\em Comput. Math. Appl.}, 80(6):1619--1632, 2020.

\bibitem[GK11]{77.5}
M.D. Gunzburger and A.~Kunoth.
\newblock Space-time adaptive wavelet methods for control problems constrained
  by parabolic evolution equations.
\newblock {\em {SIAM J. Contr. Optim.}}, 49(3):1150--1170, 2011.

\bibitem[GN16]{75.27}
M.J. Gander and M.~Neum\"uller.
\newblock Analysis of a new space-time parallel multigrid algorithm for
  parabolic problems.
\newblock {\em SIAM J. Sci. Comput.}, 38(4):A2173--A2208, 2016.

\bibitem[GR86]{75.5}
V.~Girault and P.A. Raviart.
\newblock {\em Finite element methods for {N}avier-{S}tokes equations, Theory
  and Algorithms}.
\newblock Springer-{V}erlag, Berlin, 1986.

\bibitem[LM72a]{185}
J.-L. Lions and E.~Magenes.
\newblock {\em Non-homogeneous boundary value problems and applications. {V}ol.
  {I}}.
\newblock Springer-Verlag, New York-Heidelberg, 1972.
\newblock Translated from the French by P. Kenneth, Die Grundlehren der
  mathematischen Wissenschaften, Band 181.

\bibitem[LM72b]{185.1}
J.-L. Lions and E.~Magenes.
\newblock {\em Non-homogeneous boundary value problems and applications. {V}ol.
  {II}}.
\newblock Springer-Verlag, New York-Heidelberg, 1972.
\newblock Translated from the French by P. Kenneth, Die Grundlehren der
  mathematischen Wissenschaften, Band 182.

\bibitem[LMN16]{169.05}
U.~Langer, S.E. Moore, and M.~Neum\"uller.
\newblock Space-time isogeometric analysis of parabolic evolution problems.
\newblock {\em Comput. Methods Appl. Mech. Engrg.}, 306:342--363, 2016.

\bibitem[NS19]{234.7}
M.~Neum\"{u}ller and I.~Smears.
\newblock Time-parallel iterative solvers for parabolic evolution equations.
\newblock {\em {SIAM J. Sci. Comput.}}, 41(1):C28--C51, 2019.

\bibitem[PP19]{pp19}
C.-M. Pfeiler and D.~Praetorius.
\newblock D\"orfler marking with minimal cardinality is a linear complexity
  problem, 2019.
\newblock arXiv:1907.13078.

\bibitem[RS18]{243.867}
N.~{Rekatsinas} and R.~{Stevenson}.
\newblock An optimal adaptive tensor product wavelet solver of a space-time
  fosls formulation of parabolic evolution problems.
\newblock {\em Adv. Comput. Math.}, 2018.

\bibitem[Sie11]{249.025}
K.G. Siebert.
\newblock A convergence proof for adaptive finite elements without lower bound.
\newblock {\em IMA J. Numer. Anal.}, 31(3):947--970, 2011.

\bibitem[SS09]{247.15}
Ch. Schwab and R.P. Stevenson.
\newblock A space-time adaptive wavelet method for parabolic evolution
  problems.
\newblock {\em Math. Comp.}, 78:1293--1318, 2009.

\bibitem[SS17]{247.155}
Ch. Schwab and R.P. Stevenson.
\newblock Fractional space-time variational formulations of {(Navier)-Stokes}
  equations.
\newblock {\em SIAM J. Math. Anal.}, 49(4):2442--2467, 2017.

\bibitem[Ste08]{249.87}
R.P. Stevenson.
\newblock The completion of locally refined simplicial partitions created by
  bisection.
\newblock {\em Math. Comp.}, 77:227--241, 2008.

\bibitem[Ste14]{249.96}
R.P. Stevenson.
\newblock First-order system least squares with inhomogeneous boundary
  conditions.
\newblock {\em IMA J. Numer. Anal.}, 34(3):863--878, 2014.

\bibitem[Ste15]{249.2}
O.~Steinbach.
\newblock Space-{T}ime {F}inite {E}lement {M}ethods for {P}arabolic {P}roblems.
\newblock {\em Comput. Methods Appl. Math.}, 15(4):551--566, 2015.

\bibitem[Sto19]{storn19}
J.~Storn.
\newblock {\em Topics in least-squares and discontinuous {P}etrov-{G}alerkin
  finite element analysis}.
\newblock PhD thesis, Humboldt-Universit{\"a}t zu Berlin, 2019.

\bibitem[SvV19]{249.97}
R.P. Stevenson and R.~van Veneti\"{e}.
\newblock Uniform preconditioners for problems of negative order.
\newblock {\em Math. Comp.}, 2019.

\bibitem[SW20]{249.99}
R.P. Stevenson and J.~Westerdiep.
\newblock Stability of {G}alerkin discretizations of a mixed space-time
  variational formulation of parabolic evolution equations.
\newblock {\em {IMA J. Numer. Anal.}}, 2020.

\bibitem[SZ90]{247.2}
L.~R. Scott and S.~Zhang.
\newblock Finite element interpolation of nonsmooth functions satisfying
  boundary conditions.
\newblock {\em Math. Comp.}, 54(190):483--493, 1990.

\bibitem[SZ18]{249.3}
O.~Steinbach and M.~Zank.
\newblock Coercive space-time finite element methods for initial boundary value
  problems.
\newblock Berichte aus dem {I}nstitut f{\uumlaut}r {A}ngewandte {M}athematik,
  {B}ericht 2018/7, Technische Universit{\aumlaut}t Graz, 2018.

\bibitem[UP14]{299}
K.~Urban and A.~T. Patera.
\newblock An improved error bound for reduced basis approximation of linear
  parabolic problems.
\newblock {\em Math. Comp.}, 83(288):1599--1615, 2014.

\bibitem[VR18]{310.6}
I.~Voulis and A.~Reusken.
\newblock A time dependent {S}tokes interface problem: {W}ell-posedness and
  space-time finite element discretization.
\newblock {\em ESAIM Math. Model. Numer. Anal.}, 52(6):2187--2213, 2018.

\bibitem[Wlo82]{314.9}
J.~Wloka.
\newblock {\em Partielle {D}ifferentialgleichungen}.
\newblock B. G. Teubner, Stuttgart, 1982.
\newblock Sobolevr\"aume und Randwertaufgaben.

\end{thebibliography}
\end{document}